\documentclass[a4paper,12pt]{amsart}

\usepackage[top=30truemm,bottom=25truemm,left=35truemm,right=35truemm]{geometry} 

\title[Automorphism groups of leafless metric graphs]
{Upper bounds of orders of automorphism groups of leafless metric graphs}

\author{Yusuke Nakamura}
\address{Graduate School of Mathematical Sciences, 
the University of Tokyo, 3-8-1 Komaba, Meguro-ku, Tokyo 153-8914, Japan.}
\email{nakamura@ms.u-tokyo.ac.jp}

\author{Song JuAe}
\address{Tokyo Metropolitan University 1-1 Minami-Ohsawa, Hachioji, Tokyo, 192-0397, Japan.}
\email{song-juae@ed.tmu.ac.jp}

\subjclass[2020]{Primary 14T20; Secondary 05C99}

\keywords{tropical curves, automorphisms of graphs, tropical analogue of the theorem of Hurwitz}

\usepackage{amsmath,amssymb,amscd}
\usepackage{amsthm}
\usepackage[dvips]{graphicx}
\usepackage{comment}
\usepackage{subfigure}
\usepackage{color}

\newtheorem{dfn}{Definition}
\newtheorem{thm}[dfn]{Theorem}
\newtheorem{prop}[dfn]{Proposition}
\newtheorem{lemma}[dfn]{Lemma}

\newtheorem{claim}[dfn]{Claim}
\newtheorem{rem}[dfn]{Remark}

\def \Gamma {\varGamma}

\begin{document}

\maketitle

\begin{abstract}
We prove a tropical analogue of the theorem of Hurwitz: a leafless metric graph of genus $g \ge 2$ has at most $12$ automorphisms when $g = 2$; $2^g g!$ automorphisms when $g \ge 3$.
These inequalities are optimal; for each genus, we give all metric graphs which have the maximum numbers of automorphisms.
The proof is written in terms of graph theory.
\end{abstract}

\section{Introduction}
	\label{section1}

\newcounter{num}
\setcounter{num}{4}
In classical algebraic geometry, the theorem of Hurwitz says that an algebraic curve of genus $g \ge 2$ has at most $84(g - 1)$ automorphisms (cf. \cite[Ex. 2.5. in Chapter \Roman{num}]{Hartshorne}).
In this paper, we will prove a tropical analogue of this theorem.

Tropical geometry is, roughly speaking, an algebraic geometry over the tropical semifield $\boldsymbol{T} := (\boldsymbol{R} \cup \{ -\infty \}, {\rm max}, +)$ (cf. \cite{MS} for an introduction to tropical geometry).
The process of passing from classical arithmetic (resp. classical algebro-geometric objects) to tropical arithmetic (resp. tropical objects) is referred to as tropicalization.

Tropicalizations of algebraic curves have abstract tropical curve structures.
An abstract tropical curve is an extension of a metric graph.
A \textit{metric graph} is defined as the underlying metric space of the pair of an unweighted, undirected, finite, connected nonempty multigraph $G$ which may have loops and a length function $l : E_G \to \boldsymbol{R}_{>0}$ with an identification of each edge $e$ of $G$ with the closed interval $[0, l(e)]$, where $E_G$ denotes the set of edges of $G$.
For an abstract tropical curve $\Gamma$, we allow $l$ to take the value $\infty$ on only edges incident to leaves.
Each point $\infty \in [0, \infty]$ must be identified with a leaf, where $[0, \infty]$ is the one point compactification of $[0, \infty)$.
If $l$ takes the value $\infty$, then $\Gamma$ is no longer a metric space, but a topological space (cf. \cite{JuAe} for more details).

An abstract tropical curve $\Gamma$ has a \textit{genus} in the usual topological sense.
It coincides with $\# E_G - \# V_G + 1$ for any pair $(G, l)$ defining $\Gamma$, where $V_G$ denotes the set of vertices of $G$.
Nonnegative integer $\# E_G - \# V_G + k$ for a graph $G$ with $k$ connected components is sometimes called the \textit{(first) Betti number} of $G$ (cf. \cite{BCF}).

In the category of abstract tropical curves, morphisms between abstract tropical curves are (finite) harmonic morphisms.
An automorphism of a metric graph $\Gamma$, i.e., a finite harmonic morphism $\Gamma \to \Gamma$ of degree one is an isometry $\Gamma \to \Gamma$, and vice versa.
For a metric graph $\Gamma$, we write its automorphism group as $\operatorname{Aut}(\Gamma)$, which coincides with the isometry group of $\Gamma$.

The following is our main theorem:

\begin{thm}
	\label{thm:tropaut}
Let $\Gamma$ be a leafless metric graph of genus $g \ge 2$.
Then the inequality
\[
\# \operatorname{Aut}(\Gamma) \le
\begin{cases}2^g g! & \text{if $g \ge 3$,}\\
 12 & \text{if $g = 2$}\end{cases}
\]
holds.
Furthermore, the inequality becomes an equality if and only if $\Gamma$ is one of the following:
\begin{itemize}
\item $\Gamma$ is defined by the pair of $G_{{\rm banana}, g}$ and some $l$ whose image is one point with $g = 2, 3$.
\item $\Gamma$ is defined by the pair of $G_{{\rm bouquet}, g}$ and some $l$ whose image is one point with $g \ge 3$.
\item $\Gamma$ is defined by the pair of $G_{{\rm lollipop}, g}$ and some $l$ such that all $g$ loops have the same length and all $g$ bridges have the same length with $g \ge 3$.
\end{itemize}
\end{thm}

Here, an abstract tropical curve is \textit{leafless} if it has no one valent points (it is said to be \textit{minimal} in \cite{ABBR2}).
The \textit{valency} of a point $x$ of an abstract tropical curve is the minimum number of connected components of $U \setminus \{ x \}$ with all neighborhoods $U$ of $x$. 
For $g \ge 1$, $G_{{\rm banana}, g}$ denotes the multigraph that has only two vertices and $g + 1$ multiple edges between them.
For $g \ge 1$, $G_{{\rm bouquet}, g}$ denotes the graph that has only one vertex and $g$ loops incident to it.
For $g \ge 2$, $G_{{\rm lollipop}, g}$ denotes the graph obtained from $G_{{\rm bouquet}, g}$ by replacing the unique vertex with the star $S_g$ with $g$ leaves.

We note that there are no upper bounds of $\# \operatorname{Aut}(\Gamma)$ without the leafless condition in Theorem \ref{thm:tropaut}. 
Let $\Gamma$ be a metric graph of genus $g$, and let $v \in \Gamma$ be a point. 
Let $\Gamma^{\prime}$ be the metric graph obtained from $\Gamma$ by attaching $0$'s of $n$ copies of the intervals $[0, 1]$ to the point $v$. 
Then $\Gamma^{\prime}$ has genus $g$, and we have $\# \operatorname{Aut}(\Gamma ') \ge n!$ since $\operatorname{Aut}(\Gamma^{\prime})$ contains the symmetric group of degree $n$ as a subgroup. 

In tropical geometry, it is natural to deal only leafless ones.
There is an operation called \textit{tropical modification} which defines an equivalence relation on all abstract tropical curves. 
Two abstract tropical curves are equivalent to each other if and only if one of them is obtained from the other by a finite number of retractions that contract a leaf edge to its one point (cf. \cite{ABBR2}).
Hence for an equivalence class, all representatives have the same genus and there exists a unique leafless representative.
For this reason, in tropical geometry, one studies leafless metric graphs for some essential information of their equivalence classes (cf. \cite{Yamamoto} for higher dimension cases).

Theorem \ref{thm:tropaut} follows the following combinatorial proposition:

\begin{prop}
	\label{prop:aut}
Let $G$ be a connected leafless graph of Betti number $g \ge 2$. 
Then the inequality
\[
\# \operatorname{Aut}(G) \le
\begin{cases}2^g g! & \text{if $g \ge 3$,}\\
 12 & \text{if $g = 2$}\end{cases}
\]
holds. 
Furthermore, the inequality becomes an equality if and only if $G$ is one of the following: 
\begin{itemize}
\item[(A)] $G$ is a graph obtained from $G_{{\rm banana}, g}$ by subdividing all $g + 1$ edges into the same number ($\ge 1$) of edges with $g = 2,3$.
\item[(B)] $G$ is a graph obtained from $G_{{\rm bouquet}, g}$ by subdividing all $g$ loops into the same number ($\ge 2$) of edges with $g \ge 3$. 
\item[(C)] $G$ is a graph obtained from $G_{{\rm lollipop}, g}$ by subdividing all $g$ loops into the same number 
($\ge 2$) of edges and all $g$ bridges into a same number ($\ge 1$) of edges with $g \ge 3$.
\end{itemize}
\end{prop}

\begin{figure}[h]
\begin{center}
\includegraphics[width=0.7\columnwidth]{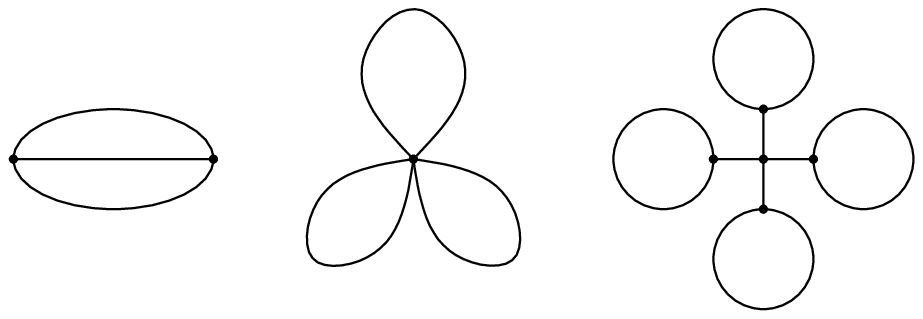}

\vspace{-1mm}
\caption{The left graph is $G_{{\rm banana}, 2}$, the center graph is  $G_{{\rm bouquet}, 3}$ and the right graph is $G_{{\rm lollipop}, 4}$. Black dots (resp. lines) stand for vertices (resp. edges).}
	\label{figure1}
\end{center}
\end{figure}

The rest of this paper is organized as follows.
In Section \ref{section2}, we prepare definitions in graph theory and three lemmas which we need later.
Section \ref{section3} gives the proofs of the two assertions above.

\section*{Acknowledgements}
The first author is partially supported by JSPS KAKENHI Grant Number 18K13384.
The second author is supported by JSPS KAKENHI Grant Number 20J11910.

\section{Preliminaries}
	\label{section2}

In this paper, a \textit{graph} $G = (V_G, E_G, \varepsilon _G)$ means an undirected finite graph 	allowing multiple edges and loops.
Here, $V_G$ is the finite set of vertices, and $E_G$ is the finite set of edges, and $\varepsilon _G : E_G \to 2^{V_G}$ 
is the map that associates an edge $e \in E_G$ with its vertices, where $2^{V_G}$ is the power set of $V_G$. 
We note that $\# ( \varepsilon _G (e) ) \in \{ 1 , 2 \}$ holds for any $e \in E_G$. 
An edge $e$ is called a \textit{loop} when $\# ( \varepsilon _G (e)) = 1$. 
A \textit{cut vertex} (resp. \textit{bridge}) of a graph $G$ is a vertex (resp. an edge) 
whose deletion increases the number of connected components of $G$. 
The \textit{degree} of a vertex is the number of edges incident to it, where a loop is counted twice.
A vertex of degree one is called a \textit{leaf}.

An \textit{automorphism} $f = (f_V, f_E)$ of a graph $G$ is 
bijective maps $f_V: V_G \to V_G$ and $f_E: E_G \to E_G$ satisfying $f_V \circ \varepsilon _G = \varepsilon _G \circ f_E$. 
$\operatorname{Aut}(G)$ denotes the set of all automorphisms of $G$. 
For a subset $S \subset V_G \sqcup E_G$, 
$\operatorname{Aut}(G)_S$ denotes the set of all $f \in \operatorname{Aut}(G)$ satisfying 
$f_V (v) = v$ and $f_E(e) = e$ for any $v \in S \cap V_G$ and $e \in S \cap E_G$. 
When $S$ consists of only one element $x$, 
we sometimes write $\operatorname{Aut}(G)_x$ instead of $\operatorname{Aut}(G)_{\{x\}}$. 
We note that according to this definition, $\# \operatorname{Aut}(G_{{\rm bouquet}, 1})=1$.

\begin{lemma}\label{lem:simplify}
Let $G$ be a connected leafless graph.
Let $x$ be a vertex and let $e_1, \ldots , e_d$ be all the edges incident to $x$. 
Let $G^{\prime}$ be the subgraph of $G$ obtained by removing $x, e_1, \ldots, e_d$ from $G$. 
Let $G^{\prime \prime}$ be the maximum subgraph of $G^{\prime}$ in which all vertices have degree at least two. 
We set 
\[
S = \left( E_{G} \setminus E_{G^{\prime \prime}} \right) \sqcup 
\bigcup _{e \in E_{G} \setminus E_{G^{\prime \prime}}} \varepsilon _G (e).
\]
Then the following hold. 
\begin{enumerate}
\item[(1)] $\operatorname{Aut}(G)_{\{ x,e_1, \ldots , e_d \}} = \operatorname{Aut}(G)_S$. 
\item[(2)] Suppose that $G^{\prime \prime} \not = \varnothing$. Then there exists an injective map 
$\operatorname{Aut}(G)_{\{ x,e_1, \ldots , e_d \}} \to \operatorname{Aut}(G^{\prime \prime})_y$ 
for some vertex $y \in V_{G^{\prime \prime}}$. 
\end{enumerate}
\end{lemma}

\begin{proof}
If $G$ is trivial, then the assertions are clear.
Assume that $G$ is not trivial.
$G^{\prime \prime}$ is obtained from $G^{\prime}$ by repeating the following two operations: 
\begin{itemize}
\item[(a)] Removing a vertex $v$ of degree zero. 
\item[(b)] Removing a leaf $v$ and the edge $e$ incident to $v$. 
\end{itemize}
Let
\[
G^{\prime} = G_0 \supset G_1 \supset \cdots \supset G_c = G^{\prime \prime}
\]
be its process. 
Then, corresponding to the two operations above, the following hold: 
\begin{itemize}
\item[(a)] $E_{G_{i+1}} = E_{G_i}$ and $V_{G_{i}} \setminus V_{G_{i+1}} = \{ v \}$. 
\item[(b)] $E_{G_{i}} \setminus E_{G_{i+1}} = \{ e \}$ and $V_{G_{i}} \setminus V_{G_{i+1}} = \{ v \}$. 
\end{itemize} \noindent
We set 
\[
S_i = \left( E_{G} \setminus E_{G_i} \right) \sqcup 
\bigcup _{e \in E_{G} \setminus E_{G_i}} \varepsilon _G (e). 
\]
Since $G$ is connected and nontrivial, $d \ge 1$ holds.
Hence $S_0 \supset \{ x, e_1, \ldots, e_d \}$.
By the definition of $S_0$, we have $\operatorname{Aut}(G)_{\{ x, e_1, \ldots, e_d \}} = \operatorname{Aut}(G)_{S_0}$.
Thus, to prove (1), it is sufficient to show 
$\operatorname{Aut}(G)_{S_i} = \operatorname{Aut}(G)_{S_{i+1}}$ for each $0 \le i \le c-1$. 

The assertion is clear for case (a). Suppose (b). 
Let $u$ be the end point of $e$ that is not $v$. 
Since all vertices of $G$ have degree at least two, it follows that $v \in S_i$.
Furthermore, any edge of $G$ incident to $v$ other than $e$ is contained in $S_i$. 
Therefore any $f \in \operatorname{Aut}(G)_{S_i}$ fixes $e$ and hence $u$ too. 
Thus we have $\operatorname{Aut}(G)_{S_i} = \operatorname{Aut}(G)_{S_{i+1}}$, which completes the proof of (1). 

By construction, it follows that $S \cap V_{G^{\prime \prime}} \not = \varnothing$ when $G^{\prime \prime} \not= \varnothing$. 
Pick any $y \in S \cap V_{G^{\prime \prime}}$. 
Then the restriction map $\operatorname{Aut}(G)_S \to \operatorname{Aut}(G^{\prime \prime})_y$ is induced and it is injective. 
\end{proof}

Let $G$ be a graph and $S \subset E_G$. 
Then the \textit{contraction} $G/S$ is a graph defined as follows:
\[
V_{G/S} = V_G / \sim, \qquad E_{G/S} = E_G \setminus S, \qquad \varepsilon_{G/S} : E_{G/S} \to 2^{V_{G/S}},
\]
where $\sim$ is the minimum relation satisfying $x \sim y$ for any $x , y \in V_G$ with 
$\{x, y\} \in \varepsilon _G (S)$, and 
$\varepsilon_{G/S}$ is the composition $E_{G/S} \to E_G \xrightarrow{\varepsilon _G} 2^{V_G} \to 2^{V_{G/S}}$.

\begin{lemma}
	\label{lem:bridge}
Let $G$ be a graph and $S \subset E_G$. Let $G/S$ be its contraction. Then the following hold. 

\begin{enumerate}
\item[{\upshape (1)}]\ If $S$ is $\operatorname{Aut} (G)$-invariant, 
there exists a group homomorphim $\operatorname{Aut} (G) \to \operatorname{Aut} (G/S)$ which commutes with 
the projection $V_G \to V_{G/S}$ and the inclusion $E_{G/S} \to E_G$.

\item[{\upshape (2)}]\ Let $S$ be the set of all bridges of $G$. Then, the Betti number of $G/S$ is equal to that of $G$. 

\item[{\upshape (3)}]\ In addtion to (2), we assume that $G$ is leafless.
Then the group homomorphism $\operatorname{Aut} (G) \to \operatorname{Aut} (G/S)$ is injective.
\end{enumerate}
\end{lemma}

\begin{proof}
We shall prove (1). 
Let $f = (f_V, f_E) \in \operatorname{Aut}(G)$. 
Since $S$ is $\operatorname{Aut} (G)$-invariant, $f_V$ and $f_E$ induce the bijective maps 
$f^{\prime} _V$ on $V_{G/S} = V_G / \sim$ and $f^{\prime} _E$ on $E_{G/S} = E_G \setminus S$, respectively. 
Since they satisfy $f^{\prime} _E \circ \varepsilon _{G/S} = \varepsilon _{G/S} \circ f^{\prime} _V$, 
we have $f^{\prime} = (f^{\prime} _V, f^{\prime} _E) \in \operatorname{Aut}(G/S)$. 
Then we obtain a group homomorphim $\operatorname{Aut} (G) \to \operatorname{Aut} (G/S);\ f \mapsto f^{\prime}$, which proves (1).

Since $S$ forms the union of trees, we have $\# V_G - \# V_{G/S} = \# E_G - \# E_{G/S}$, which proves (2).

We shall prove (3). 
Let $f \in \operatorname{Aut}(G)$ be an element whose image in $\operatorname{Aut}(G/S)$ is the identity map. 
Let $H$ be the subgraph of $G$ such that $E_H = S$ and $V_H = \bigcup _{e \in S} \varepsilon _G (e)$. 
By assumption, $f \in \operatorname{Aut}(G)_F$ holds for $F = (V_G \setminus V_H) \sqcup (E_G \setminus E_H)$. 
Let $v$ be a leaf of $H$ and let $e$ be the edge of $H$ incident to $v$. 
First, we shall see that $f$ fixes $v$. 
Since $v$ has degree at least two in $G$, there exists an edge $e_1 \in E_G \setminus E_H$ 
such that $e_1 \not = e$ and $v \in \varepsilon _G (e_1)$. 
If $e_1$ is a loop, then $f$ fixes $v$ because $f$ fixes $e_1$. 
Suppose that $e_1$ is not a loop. 
Let $v^{\prime}$ be the end point of $e_1$ that is not $v$. 
Since $f$ fixes $e_1$, it follows that $f_V(v) = v$ or $f_V(v) = v^{\prime}$. 
Since $S$ does not contain a circuit, it follows that $v \not \sim v^{\prime}$. 
On the other hand, since the image of $f$ in $\operatorname{Aut}(G/S)$ is the identity map, we have $f_V(v) \sim v$. 
Therefore $f$ fixes $v$. 
Since $f$ fixes all edges incident to $v$ except $e$, it also fixes $e$. 
Let $v^{\prime \prime}$ be the end point of $e$ that is not $v$. 
Then $f$ also fixes $v^{\prime \prime}$. 
Let $H'$ be the subgraph of $H$ obtained by removing $e$ and $v$. 
Then we have showed that 
$f \in \operatorname{Aut}(G)_{F^{\prime}}$ for $F^{\prime} = (V_G \setminus V_{H^{\prime}}) \sqcup (E_G \setminus E_{H^{\prime}})$.
Since $H$ is a finite union of trees, we can see $f \in \operatorname{Aut}(G)_{V_G \sqcup E_G} = \{ {\rm id}_G \}$ by induction, 
which completes the proof. 
\end{proof}

\begin{lemma}
	\label{lem:bikkuri}
\begin{enumerate}
\item[(1)] Let $l$, $m$ and $n$ be positive integers with $m,n \ge l$. 
Then we have 
\[
m! n! \le l! (m+n-l)!. 
\]
\item[(2)] Let $n_1, \ldots , n_c$ be positive integers. Then we have 
\[
(n_1 !) (n_2 !) \dotsm (n_c !) \le \left( - c + 1 + \sum_{ i = 1}^{c} n_i \right) !.
\]
The inequality becomes an equality if and only if all but one of $n_1, \ldots, n_c$ are equal to one.
\end{enumerate}
\end{lemma}

\begin{proof}
(1) is straightforward. By induction on $c$, (2) is reduced to (1) for $l=1$. 
\end{proof}

\section{Main results}
	\label{section3}

In this section, we will prove the two assertions in Section 1.

The following proposition is valid also for $g=1$ in contrast to Proposition \ref{prop:aut}.

\begin{prop}
	\label{prop:fixedpt}
Let $G$ be a connected leafless graph of Betti number $g$ and $x$ a vertex of $G$.
Then the inequality
\[
\# \operatorname{Aut}(G)_x \le 2^g g!
\]
holds.
\end{prop}

\begin{proof}
If $g = 0$, then $G$ is trivial, and the assertion is clear. 
Therefore we may assume that $g \ge 1$. 
Furthermore, by Lemma \ref{lem:bridge}, we may assume that $G$ has no bridges.

Let 
\[
G \setminus \{ x \} = U_1 \sqcup \cdots \sqcup U_k
\]
be the decomposition to the connected components. 
Let $G_i = U_i \cup \{ x \}$ be the subgraph of $G$. 
Let $g_i$ be the Betti number of $G_i$.
Then $g = \sum_{i=1}^{k}g_i$ and $g_i \ge 1$ hold. 
Furthermore, since $G$ has no bridges, each $G_i$ is leafless.

First, we treat the case when $k = 1$. 
\begin{claim}
	\label{claim}
When $k = 1$, the inequality
\[
\# \operatorname{Aut}(G)_x \le 2^{g - d + 1} d! (g - d + 1)!
\]
holds, where $d$ is the number of edges incident to $x$. 
In particular, Proposition \ref{prop:fixedpt} holds when $k=1$. 
\end{claim}

\begin{proof}
Let $e_1, \ldots, e_d$ be all the edges incident to $x$.
Let $S_d = S\{ e_1, \ldots, e_d \}$ denote the permutation group of $\{ e_1, \ldots, e_d \}$.
Then a group homomorphism $s$ is defined by
\[
s : \operatorname{Aut}(G)_x \to S_d; \qquad f \mapsto \left( \begin{smallmatrix} e_1 & e_2 & \cdots & e_d \\ f(e_1) & f(e_2) & \cdots & f(e_d) \end{smallmatrix} \right).
\]
Thus
\[
\# \operatorname{Aut}(G)_x \le \# S_d \cdot \# \operatorname{Ker}(s) = d! \cdot \# \operatorname{Aut} (G)_{\{x, e_1, \ldots, e_d \}}
\]
hold. 
We define the subgraph $G^{\prime}$ of $G$ by removing $x, e_1, \ldots, e_d$ from $G$.
Let $G^{\prime \prime}$ be the maximum subgraph of $G^{\prime}$ in which all vertices have degree at least two. 
We note that $G^{\prime}$ and $G^{\prime \prime}$ are connected since $k=1$. 
Furthermore, $G^{\prime} \not = \varnothing$ holds unless $G = G_{{\rm bouquet}, 1}$. 
If $G^{\prime} \not = \varnothing$, then the Betti number of $G^{\prime}$ is equal to $g-d+1$. 
In particular, we have $g-d+1 \ge 0$.

Suppose $d = g + 1$. 
Then $G^{\prime}$ is a tree, and it follows that $G^{\prime \prime} = \varnothing$. 
Then by Lemma \ref{lem:simplify}, 
\[
\# \operatorname{Aut} (G) _{\{ x, e_1, \ldots, e_d\}} = \# \operatorname{Aut} (G) _{V_G \sqcup E_G} = 1, 
\]
and hence
\[
\# \operatorname{Aut} (G)_x \le d! \cdot \# \operatorname{Aut} (G)_{\{x, e_1, \ldots, e_d \}} = (g + 1)! \le 2^g g!. \tag{\ref{claim}.1}
\]

Suppose that $g - d + 1 \ge 1$. 
If $G = G_{{\rm bouquet}, 1}$, then we have $\# \operatorname{Aut}(G)_x = 1 < 2 = 2^{g - d + 1} d! (g - d + 1)!$. 
In what follows, we suppose $G^{\prime \prime} \not = \varnothing$. 
Then we have $d \ge 2$ and the Betti number of $G^{\prime \prime}$ is equal to $g' := g-d+1$. 
Then by Lemma \ref{lem:simplify}(2), there exists a vertex $y$ of $G^{\prime \prime}$ such that
\[
\# \operatorname{Aut}(G)_{\{x, e_1, \ldots, e_d \}} \to \# \operatorname{Aut}(G^{\prime \prime})_y; \qquad f \mapsto f|_{G^{\prime \prime}}
\]
is injective.
Thus, we have 
\[
\# \operatorname{Aut}(G)_{\{x, e_1, \ldots, e_d \}} \le \# \operatorname{Aut}(G^{\prime \prime})_y.
\]
By induction on the Betti number, we may assume that 
\[
\# \operatorname{Aut}(G^{\prime \prime})_y \le 2^{g^{\prime}} g^{\prime}! = 2^{g - d + 1} (g - d + 1)!.
\]
Therefore, we have
\[
\# \operatorname{Aut}(G)_x \le 2^{g - d + 1} d! (g - d + 1)! < 2^g g!.
\]
Here, the second inequality follows from Lemma \ref{lem:bikkuri} and $d \ge 2$. 
\end{proof}

Suppose $k \ge 2$. 
Then by $f \in \operatorname{Aut}(G)_x$, each $G_i$ is mapped to some $G_j$.
It induces the group homomorphism
\[
t: \operatorname{Aut}(G)_x \to S_k = S\{ G_1, \ldots, G_k \}; \qquad f \mapsto \left( \begin{smallmatrix} G_1 & G_2 & \cdots & G_k \\ f(G_1) & f(G_2) & \cdots & f(G_k) \end{smallmatrix} \right).
\]
Since the kernel $\operatorname{Ker} (t)$ consists of all $f \in \operatorname{Aut}(G)_x$ such that $f(G_i) = G_i$ for each $i$,
\[
\# \operatorname{Aut}(G)_x \le \# S_k \cdot \# \operatorname{Ker} (t) \le k! \cdot \# \operatorname{Aut}(G_1)_x \dotsm \# \operatorname{Aut}(G_k)_x
\]
hold.
Then, by induction on the Betti number, we have 
\begin{align*}
\# \operatorname{Aut}(G)_x \tag{\ref{prop:fixedpt}.1}
&\le k! (2^{g_1} g_1 !) \dotsm (2^{g_k} g_k !) \\ 
&= 2^g (k!) (g_1 !) \dotsm (g_k !)\\ \tag{\ref{prop:fixedpt}.2}
&\le 2^g \left( - k - 1 + 1 + k + \sum _{i=1} ^k g_i \right) ! \\
& = 2^g g!.
\end{align*}
The last inequality follows from Lemma \ref{lem:bikkuri}.
We complete the proof of Proposition \ref{prop:fixedpt}. 
\end{proof}

\begin{rem}
	\label{rem:equality1}
\upshape{
In Proposition \ref{prop:fixedpt}, the equality $\# \operatorname{Aut}(G)_x = 2^g g!$ holds if and only if the pair of $G$ and $x$ is one of the following: 
\begin{enumerate}
\item $G$ is trivial and $x$ is the unique vertex. 
\item $G$ is a subdivision of $G_{{\rm banana}, 1}$ and $x$ is any vertex. 
\item $G$ is one of the graphs in (B) of Proposition \ref{prop:aut} with $g \ge 2$ and $x$ is the unique cut vertex. 
\item $G$ is one of the graphs in (C) of Proposition \ref{prop:aut} with $g \ge 2$ and $x$ is the unique cut vertex of $S_g$. 
\end{enumerate}

First, it is clear that $G$ and $x$ in the above list satisfy the equality $\# \operatorname{Aut}(G)_x = 2^g g!$. 
In what follows, we shall see the inverse implication. 

Suppose that $G$ and $x$ satisfy the equality. 
If $g = 0$, then $G$ must be (1). 
Suppose that $g \ge 1$ and $G$ is bridgeless. 
When $k = 1$, we have $g = 1$ by the second inequality in (\ref{claim}.1) in the proof of Claim \ref{claim}, and the pair $(G, x)$ is confined to (2). 
When $k \ge 2$, we note that the inequality (\ref{prop:fixedpt}.2) becomes an equality only if $g_1 = \cdots = g_k = 1$. 
In this case, $G$ must be a subdivision of $G_{{\rm bouquet}, g}$. 
For such $G$, it is clear that only the pairs $(G, x)$ as in (3) satisfy the equality. 
Finally, suppose that $g \ge 1$ and $G$ has a bridge. 
Let $S$ be the set of all bridges, and let $G/S$ be the contraction. 
Then, from what we have already seen, $G/S$ must be one of the graphs in (3). 
Therefore, the pair $(G, x)$ turns out to be one of the pairs in (4).
}
\end{rem}

Now we can prove Proposition \ref{prop:aut}.

\begin{proof}[Proof of Proposition \ref{prop:aut}]
First, it is clear that $G$ in the list (A), (B), and (C) satisfies 
the equality \[
\# \operatorname{Aut}(G) = 
\begin{cases}
2^g g! & \text{if $g \ge 3$,} \\
12 & \text{if $g = 2$}. \tag{$\star$} 
\end{cases}
\]

Suppose that $G$ has a bridge. 
Let $S$ be the set of all bridges, and let $G/S$ be the contraction. 
Then by Lemma \ref{lem:bridge}, we have $\# \operatorname{Aut}(G) \le \# \operatorname{Aut}(G/S)$. 
Furthermore, if $G$ satisfies the equality ($\star$) and if $G/S$ is one of the graphs in (A) and (B), 
then it is easy to see that $G$ is one of the graphs in (C). 

Suppose that $G$ has a loop. 
Let $G'$ be the subdivision of $G$ obtained by adding one vertex to each loop of $G$. 
Then $G'$ is a loopless graph and we have $\# \operatorname{Aut}(G) < \# \operatorname{Aut}(G')$. 

By the discussion above, it is sufficient to show
the inequality in Proposition \ref{prop:aut} for $G$ with the following additional condition: 
\begin{itemize}
\item $G$ is loopless and bridgeless. 
\end{itemize}
Furthermore, for such $G$, it is sufficient to show the equality ($\star$) holds only if 
$G$ is one of the graphs in (A) and (B). 
In what follows, we suppose that $G$ is loopless and bridgeless. 

Let $d$ be the maximum degree of the vertices of $G$ and let $l$ be the number of vertices of $G$ of degree $d$. 
Since $g \ge 2$, we have $d \ge 3$. 
Furthermore, by the hand-shaking lemma, we have
\[
2 \# E_G \ge d l + 2 (\# V_G - l), 
\]
and hence 
\begin{align*}
g &= \# E_G - \# V_G + 1\\
& \ge \frac{d l}{2} + \# V_G - l - \# V_G + 1\\
&= \left( \frac{d}{2} - 1 \right) l + 1. 
\end{align*}
As $d \ge 3$, we have $l \le 2g - 2$. 

Let $x_1, \ldots, x_l$ be all the vertices of $G$ of degree $d$. 
Set $x = x_1$. 
If $l = 1$, then $\operatorname{Aut} (G) = \operatorname{Aut}(G)_x$ holds and the assertion follows 
from Proposition \ref{prop:fixedpt} and Remark \ref{rem:equality1}.
In what follows, we assume $l \ge 2$, and in particular, $g \ge d-1$. 
The group homomorphism 
\[
s: \operatorname{Aut}(G) \to S_l = S \{ x_1, \ldots, x_l \}; \quad 
f \mapsto \left( \begin{smallmatrix} x_1 & x_2 & \cdots & x_l \\ f(x_1) & f(x_2) & \cdots & f(x_l) \end{smallmatrix} \right)
\]
induces an injective map on the left cosets 
\[
\bigl \{ f  \operatorname{Aut}(G)_x \ \big| \ f \in \operatorname{Aut}(G) \bigr \} \to 
\bigl \{ \sigma  S \{ x_2, \ldots, x_l \} \ \big| \ \sigma \in S \{ x_1, \ldots, x_l \} \bigr \}. 
\]
Therefore we have 
\[
\# \operatorname{Aut} (G) \le l \cdot \# \operatorname{Aut}(G)_x. 
\]

Let $G \setminus \{ x \} = U_1 \sqcup \cdots \sqcup U_k$ be the decomposition to the connected components of $G \setminus \{ x \}$. 
Let $G_i = U_i \cup \{ x \}$ be the subgraph of $G$, and let $g_i$ be the Betti number of $G_i$.
Then $g = \sum_{i = 1}^k g_i$ and $g_i \ge 1$ hold. 
Furthermore, each $G_i$ is loopless since $G$ is bridgeless. 

We deal with two cases where $k = 1$ and where $k \ge 2$ separately. 

\vspace{3mm}

\noindent
\underline{\textbf{Case 1}}\ Suppose $k=1$.

\vspace{2mm}

\noindent
\underline{\textbf{Case 1-1}}\ Suppose $d = g + 1$. 
Then we have
\[
g \ge  \left( \frac{d}{2} - 1 \right) l + 1 = g + \left( \frac{l}{2} - 1 \right) (g - 1). 
\]
Therefore, we have $l = 2$ since $g \ge 2$ and $l \ge 2$. 
By Claim \ref{claim}, we have
\[
\# \operatorname{Aut} (G) \le l \cdot \# \operatorname{Aut}(G)_x \le 2 d! = 2 \cdot (g + 1)!.
\]
Here, we note that $d$ equals the number of the edges incident to $x$ since $G$ is loopless.
 
When $g \ge 3$, we have $2 \cdot (g + 1)! \le 2^g g!$ and get the desired inequality 
$\# \operatorname{Aut} (G) \le 2^g g!$. 
We note that the equality $\# \operatorname{Aut} (G) = 2^g g!$ holds only if $(g, d, l) = (3,4,2)$. 
Under the assumption that $k = 1$, only subdivisions of $G_{{\rm banana}, 3}$ may satisfy $(g,d,l)=(3,4,2)$.
Hence $G$ must be one of the graphs in (A) with $g = 3$ if $\# \operatorname{Aut} (G) = 48 = 2^g g!$ hold.

When $g = 2$, we have $d = 3$. 
In this case, since $G$ is bridgeless, $G$ is a subdivision of $G_{\rm banana, 2}$ and we have $\# \operatorname{Aut} (G) \le 12 = 2^g g!$. 
The equality $\# \operatorname{Aut} (G) = 12$ holds if and only if $G$ is one of the graphs in (A) with $g = 2$. 

\vspace{2mm}

\noindent
\underline{\textbf{Case 1-2}}\ 
Suppose $d < g + 1$. 
By Claim \ref{claim}, we have
\begin{align*}
\# \operatorname{Aut} (G) 
&\le l \cdot \# \operatorname{Aut}(G)_x \\
&\le l \cdot 2^{g - d + 1} d! (g - d + 1)! \\
&\le (2g - 2) 2^{g - d + 1} d! (g - d + 1)!.
\end{align*}
Thus, it is sufficient to show that
\[
d! (g - d + 1)! (2g - 2) \le 2^{d - 1} g!.
\]

If $g - d + 1 = 1$, then the inequality is equivalent to $g - 1 \le 2^{g - 2}$ and it holds since $g = d \ge 3$. 
Therefore we get the desired inequality $\# \operatorname{Aut} (G) \le 2^g g!$. 
We note that the equality $\# \operatorname{Aut} (G) = 2^g g!$ holds only if $g$ and $l$ satisfy
\[
l = 2g - 2, \quad g-1 = 2^{g-2}. 
\]
Since $g = d \ge 3$, $(g,d,l) = (3,3,4)$ is the only case. 
In this case, $G$ is one of the following graphs. 
\begin{itemize}
\item A subdivision of $H_1$ in Figure \ref{figure:H_1,H_2}.
\item A subdivision of $H_2$ in Figure \ref{figure:H_1,H_2}.
\end{itemize}
We have $\# \operatorname{Aut} (G) \le \# \operatorname{Aut} (H_1) = 4! < 2^g g!$ in the formar case, and 
$\# \operatorname{Aut} (G) \le \# \operatorname{Aut} (H_2) = 16 < 2^g g!$ in the latter case. 
Therefore, the equality $\# \operatorname{Aut} (G) = 2^g g!$ never holds in this case. 
\begin{figure}[h]
\begin{center}
\includegraphics[width=0.6\columnwidth]{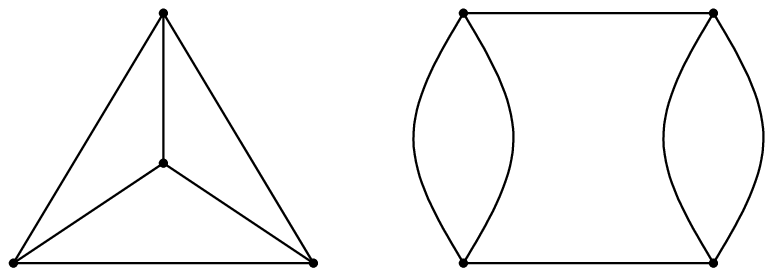}

\vspace{1mm}
\caption{We call the left graph $H_1$ and the right graph $H_2$. Black dots (resp. lines) stand for vertices (resp. edges).}
	\label{figure:H_1,H_2}
\end{center}
\end{figure}

If $g - d + 1 \ge 2$, we have
\begin{align*}
\frac{2^{d - 1}}{2g - 2} \cdot \frac{g!}{d! (g - d + 1)!} 
&\ge \frac{4}{2g - 2} \cdot \frac{g!}{d! (g - d + 1)!} \\
&= \frac{2}{(g + 1) (g - 1)} \left( \begin{matrix} g + 1 \\ d \end{matrix} \right) \\
& > 1.  
\end{align*}
Here the first inequality follows because $d \ge 3$. 
The second inequality follows from the fact that
\[
\left( \begin{matrix} g + 1 \\ d \end{matrix} \right) 
\ge \left( \begin{matrix} g + 1 \\ 2 \end{matrix} \right) = \frac{(g + 1) g}{2} > \frac{(g + 1)(g-1)}{2}, 
\]
which is obtained by $d \ge 2$ and $g - d + 1 \ge 2$.
Therefore we get the strict inequality $\# \operatorname{Aut} (G) < 2^g g!$ in this case.

\vspace{3mm}

\noindent
\underline{\textbf{Case 2}}\ Suppose $k \ge 2$. 

First, we see that $d \ge 4$, and especially $2 \le l \le g - 1$. 
Suppose the contrary that $d = 3$. 
Let $e_1$, $e_2$, and $e_3$ be the edges incident to $x$. 
As $k \ge 2$, there exists $G_i$ such that $E_{G_i}$ contains just one of $e_1$, $e_2$, and $e_3$. 
Such edge should be a bridge, a contradiction. 

Let $t$ denote the group homomorphism 
\[
t : \operatorname{Aut}(G)_x \to S_k = S \{ G_1, \ldots, G_k \}; \qquad f \mapsto \left( \begin{smallmatrix} G_1 & G_2 & \cdots & G_k \\ f(G_1) & f(G_2) & \cdots & f(G_k) \end{smallmatrix} \right). 
\]
Then, we have
\[
\# \operatorname{Aut}(G)_x = \# \operatorname{Im} (t) \cdot \# \operatorname{Ker} (t) \le \# \operatorname{Im} (t) \cdot \# \operatorname{Aut}(G_1)_x \dotsm \# \operatorname{Aut}(G_k)_x.
\]
Hence, by Proposition \ref{prop:fixedpt}, we have
\begin{align*}
\# \operatorname{Aut} (G) 
&\le l \cdot \# \operatorname{Aut}(G)_x \\
&\le l \cdot \# \operatorname{Im} (t) \cdot (2^{g_1} g_1 !) \dotsm ( 2^{g_k} g_k !) \\
&= l \cdot \# \operatorname{Im} (t) \cdot 2^g \cdot (g_1 !) \dotsm (g_k !). 
\end{align*}

\noindent
\underline{\textbf{Case 2-1}}\ Suppose that $g_1 = g_2 = \cdots = g_k$ does not hold. 

Since $f \in \operatorname{Aut} (G)$ maps $G_i$ to only $G_j$ such that $g_i = g_j$, 
we have $\# \operatorname{Im} (t) \le (k - 1)!$ in this case. 
Then we have 
\begin{align*}
l \cdot \# \operatorname{Im} (t) \cdot (g_1 !) \dotsm (g_k !) 
&\le (g -1) (k -1)! (g_1 !) (g_2 !) \dotsm (g_k !) \\
&\le (g -1) (g - 1)! \\
&< g!.  
\end{align*}
Here the second inequality follows from Lemma \ref{lem:bikkuri}. 
Therefore we get the strict inequality $\# \operatorname{Aut} (G) < 2^g g!$ in this case.

\vspace{2mm}

\noindent
\underline{\textbf{Case 2-2}}\ Suppose that $k \ge 3$ and $g_1 = g_2 = \cdots = g_k$.

Suppose that $x$ is fixed for any $f \in \operatorname{Aut}(G)$. 
Then $\operatorname{Aut}(G) = \operatorname{Aut}(G)_x$ holds and the assertion follows 
from Proposition \ref{prop:fixedpt} and Remark \ref{rem:equality1}. 
Therefore we may assume that $f(x) \not= x$ for some $f \in \operatorname{Aut}(G)$. 
We set $x^{\prime} = f(x)$. 
Let 
\[
G \setminus \{ x^{\prime} \} = U^{\prime}_1 \sqcup \cdots \sqcup U^{\prime}_{k^{\prime}}
\]
be the decomposition to the connected components of $G \setminus \{ x^{\prime} \}$. 
Let $G^{\prime}_i = U^{\prime}_i \cup \{ x^{\prime } \}$ be the subgraph of $G$, and let 
$g^{\prime}_i$ be its Betti number.
Since the automorphism $f$ maps $x$ to $x^{\prime}$, 
$k^{\prime}$ must be $k$ and $g'_i = \frac{g}{k}$ holds for each $i$. 
We may assume that $x^{\prime} \in V_{G_1}$. 
Suppose $x \in V_{G^{\prime}_i}$. 
Then $G_2, \ldots, G_k$ are subgraphs of $G^{\prime}_i$. 
Therefore it follows that
\[
\frac{k-1}{k}g = g_2 + \cdots + g_k \le g^{\prime}_i = \frac{g}{k},
\]
and contradicts the assumption $k \ge 3$. 

\vspace{2mm}

\noindent
\underline{\textbf{Case 2-3}}\ Suppose that $k=2$ and $g_1 = g_2$.

Let $g^{\prime} := g_1 = g_2$. 
Since $l \le g-1 = 2g^{\prime} -1$, we have
\[
l \cdot \# \operatorname{Im}(t) \cdot (g^{\prime} !)^2 \le (2 g^{\prime} - 1) \cdot 2 \cdot  (g^{\prime} !)^2 \le (2 g^{\prime}) ! = g!,
\]
which proves the desired inequality $\# \operatorname{Aut} (G) \le 2^g g!$. 
We note that the equality $\# \operatorname{Aut} (G) = 2^g g!$ holds only if $g'$ and $l$ satisfy
\[
l = 2g'-1, \quad (2 g^{\prime} - 1) \cdot 2 \cdot  (g^{\prime} !)^2 = (2 g^{\prime}) !. 
\]
Since $2g' = g \ge 3$, we have $(g', g, d, l) = (2, 4, 4, 3)$. 
We note that such $G$ has no vertices of degree three by the hand-shaking lemma. 
Therefore, in this case, $G$ turns out to be a subdivision of the graph $H$ in Figure \ref{figure:H}.
Hence we have $\# \operatorname{Aut}(G) \le \# \operatorname{Aut}(H) = 2^5 < 2^g g!$, 
and the equality $\# \operatorname{Aut} (G) = 2^g g!$ never holds in this case. 
\end{proof}

\begin{figure}[h]
\begin{center}
\includegraphics[width=0.6\columnwidth]{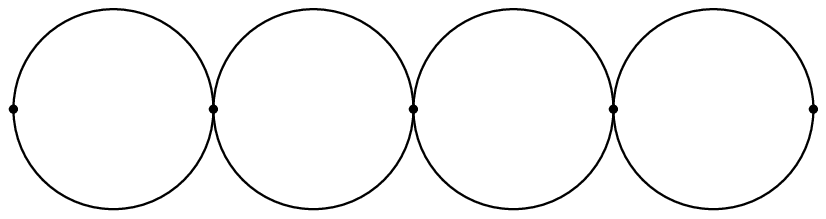}

\vspace{1.5mm}
\caption{We call the graph $H$. Black dots (resp. lines) stand for vertices (resp. edges).}
	\label{figure:H}
\end{center}
\end{figure}

Finally, we shall prove Theorem \ref{thm:tropaut}.

\begin{proof}[Proof of Theorem \ref{thm:tropaut}]
Let $V$ be the set of all points of $\Gamma$ except two valent points.
We add to $V$ all midpoints of loops.
Let $(G, l)$ be the pair defining $\Gamma$ that has the added set $V$ as its set of vertices.
Each $f \in \operatorname{Aut}(\Gamma)$ induces a permutation of the subset of $\Gamma$ corresponding to $V_G$ and that of the set of intervals of $\Gamma$ corresponding to edges of $G$.
This induces a group homomorphism $\operatorname{Aut}(\Gamma) \to \operatorname{Aut}(G)$.
Since $G$ is loopless, it is injective.
Then we have
\[
\# \operatorname{Aut}(\Gamma) \le \# \operatorname{Aut}(G).
\]
By Proposition \ref{prop:aut}, we have the desired inequality.

In each case of equality conditions of Theorem \ref{thm:tropaut}, clearly, the inequality becomes an equality.
The converse follows from Proposition \ref{prop:aut} and the fact that each $f \in \operatorname{Aut}(\Gamma)$ is an isometry $\Gamma \to \Gamma$.
\end{proof}

\end{document}